\documentclass[a4paper,11pt]{amsart}

\usepackage[utf8]{inputenc}
\usepackage{amsmath}
\usepackage{amssymb}
\usepackage{amsthm}
\usepackage{tikz}
\usepackage{tikz-cd}
\usepackage{verbatim}
\usepackage{hyperref}
\usetikzlibrary{matrix}

\DeclareMathOperator{\st}{st}

\newcommand{\F}{\mathbb{F}}
\newcommand{\N}{\mathbb{N}}
\newcommand{\e}{\mathbf{e}}

\newtheorem*{thmA}{Theorem A}
\newtheorem{lemma}{Lemma}
\newtheorem{proposition}[lemma]{Proposition}
\newtheorem{cor}[lemma]{Corollary}

\theoremstyle{definition}

\newtheorem*{definition}{Definition}

\title{Multi-GGS-groups have the congruence subgroup property}

\author[A.\ Garrido]{Alejandra Garrido}
\address{Mathematisches Institut\\
	Heinrich-Heine-Universit\"{a}t D\"{u}sseldorf\\
Universit\"{a}tsstr. 1\\
40225\\
D\"{u}sseldorf, Germany}
\email{alejandra.garrido@uni-duesseldorf.de}

\author[J.\ Uria-Albizuri]{Jone Uria-Albizuri}
\address{Basque Center of Applied Mathematics\\ Mazarredo, 14.\\ 48009 \\ Bilbao, Basque Country - Spain}
\email{juria@bcamath.org}

\date{}
\thanks{A.\ Garrido is supported by the Alexander von Humboldt Foundation.
	J.\ Uria-Albizuri acknowledges financial support from the Spanish Government, grants MTM2011-28229-C02 and
MTM2014-53810-C2-2-P, from the Basque Government, grants IT753-13 and IT974-16 and also from the Basque Goverment predoctoral grant PRE-2014-1-347. This research is also supported by the Basque Government through the BERC 2018-2021 program and by the Spanish Ministry of Economy and Competitiveness MINECO: BCAM Severo Ochoa accreditation SEV-2013-0323.
 }


\begin{document}
\begin{abstract}
We generalize the result about the congruence subgroup property for GGS-groups in \cite{AGU} to the family of multi-GGS-groups; that is, all multi-GGS-groups except the one defined by the constant vector have the congruence subgroup property. 
Even if the result remains, new ideas are needed in order to generalize the proof.
\end{abstract}

\maketitle

\section{Introduction}

Groups acting on regular rooted trees have been studied since the 1980s as examples of groups with exotic properties. 
For instance, the Gupta--Sidki group (\cite{gupta-sidki:burnside}) is well known as a particularly uncomplicated answer to the General Burnside Problem. 
It acts on the ternary rooted tree and is two-generated by a rooted automorphism and a directed one (see \cite{AGU} for terminology and notation used here).
A similar example, acting on the 4-regular rooted tree, was introduced by Grigorchuk in \cite{grigorchuk:burnside} and is now sometimes known as the ``second Grigorchuk group". 
A generalization of these two examples is provided by the family of Grigorchuk--Gupta--Sidki (GGS) groups%
\footnote{The name was coined by Baumslag and is not to be confused with Generalized Golod--Shafarevich groups, which also provide negative answers to the General Burnside Problem, but are of a very different nature.}. 
Each group in  this family acts on the $p$-regular tree, where $p\geq 3$ (but here, as in most papers where they are studied, we assume $p$ is prime)  and is also generated by a rooted and a directed automorphism. 
The latter generator is defined according to a vector in $\F_p^{p-1}$. 


The classical congruence subgroup property for linear algebraic groups has a natural analogue for groups acting on regular rooted trees.
 In this case, the principal congruence subgroups are the level stabilizers, $\st_G(n)$ for each $n\in \N$.
 Thus the congruence subgroup property is satisfied if every subgroup of finite index contains a principal congruence subgroup. In other words, if the topology given by the finite index subgroups (the profinite topology) coincides with the topology given by the level stabilizers. 

In \cite{AKT}, a family further generalizing GGS-groups was defined and studied.
They were called multi-edge spinal groups, but we prefer the term multi-GGS-groups.
Each group in the family is defined like a GGS-group, except that there are more directed generators.
In \cite{AGU} we proved that a GGS-group has the congruence subgroup property if and only if it is not defined by the constant vector. 
In this note we show that the result is still true for the whole family of multi-GGS-groups.

\begin{thmA}\label{thm:CSP}
Let $G$ be a multi-GGS-group. Then $G$ is just infinite and has the congruence subgroup property if and only if $G$ is not the GGS-group defined by the constant vector.
\end{thmA}

\noindent\textbf{Acknowledgements.} We are grateful to G.A.\ Fern\'andez-Alcober for useful comments and to the anonymous referee for a prompt report including improvement suggestions.

\section{Multi-GGS-groups}

\begin{definition}
A \emph{multi-GGS-group} $G$ is a group of automorphisms of the $p$-regular rooted tree, where $p$ is an odd prime. 
The group $G$ is generated by the rooted automorphism $a$ which acts on the first level of the tree by the cyclic permutation $(1\,\dots\, p)$
and by some finite number $r$ of directed automorphisms $b_1,\dots,b_r \in \st_G(1)$. 
Each $b_i$ is defined by a vector $\e_i=(e_{i,1},\dots,e_{i,p-1})\in \F_p^{p-1}$:
$$\psi(b_i)=(a^{e_{i,1}},\dots,a^{e_{i,p-1}},b_i)$$
and we require that the defining vectors $\e_1,\dots,\e_r$ be linearly independent. 

\end{definition}

 We denote by $\mathcal{G}$ the group generated by $a$ and $ b_1$ with $b_1$ given by a constant vector.
 (Note that all constant vectors yield the same group.)
 
Let us first mention some properties about multi-GGS-groups that will be useful in the proof of the main theorem. 
The first lemma is a collection of results from \cite{AKT}.

\begin{lemma}\label{multi-lemma}
Let $G=\langle a,b_1,\dots,b_r\rangle$ be a multi-GGS-group with defining vectors $\e_1,\dots\e_r\in\F_p^{p-1}$.
\begin{itemize}
\item[(i)] We may assume that $e_{i,1}=1$ for all $\e_i$ with $i=1,\dots,r$.\\
\item[(ii)] If $G$ is not $\mathcal{G}$ then $$\psi(\gamma_3(\st_G(1)))=\gamma_3(G)\times\overset{p}{\dots}\times\gamma_3(G).$$
\item[(iii)] $G/G'\cong C_p^{r+1}.$
\end{itemize}
\end{lemma}
 
 \begin{lemma}\label{regularbranch}
 If the multi-GGS-group $G$ is generated by $r\geq 2$ directed generators, then  $\psi(\st_G(1)')= G'\times \overset{p}{\dots}\times G'$.
  In particular, $G$ is regular branch over its commutator subgroup $G'$.  
 \end{lemma}
\begin{proof}
 Since $\psi(\st_G(1))\leq G\times\overset{p}{\dots}\times G$, we need only show the `$\geq$' inclusion in the statement.

Suppose that $b_1$ has non-symmetric defining vector $(e_{1,1},\ldots, e_{1,p-1})$; that is, $e_{1,i} \ne e_{1,p-i}$ for some $i$. Either $e_{1,i} \ne 0 $ or $e_{1,p-i} \ne 0$. Without loss of generality, suppose $e_{1,i} \ne 0$. As in the proof of Lemma \ref{multi-lemma} (i), we may assume that $i = 1$ and that $e_{1,1} = 1$ and $e_{1,p-1}=m\ne 1$.

 By the same argument as in \cite[Lemma 3.4]{Alc}, 
 $$\psi([b_1,b_1^a][b_1^{a^{-1}},b_1]^m\dots [b_1^a, b_1^{a^2}]^{m^{p-1}})\equiv ([a, b_1]^{1-m},1,\dots,1)$$
 where the congruence is modulo $\gamma_3(G)\times \overset{p}{\dots}\times \gamma_3(G).$
 Part (ii) of Lemma \ref{multi-lemma} implies that $([a, b_1]^{1-m},1,\dots,1)\in \psi(\st_G(1)')$ and therefore $([a, b_1],1,\dots,1)\in \psi(\st_G(1)')$ because $m\neq 1$.
  By a general fact about commutators, 
  $$[a^n, b_1]= [a, b_1]^{a^{n-1}}[a,b_1]^{a^{n-2}}\dots [a,b_1]^a[a,b_1]$$ for any $n\in\mathbb{Z}$.

 From this and the fact that $\psi(\st_G(1))\leq G\times \overset{p}{\dots}\times G$ is  a subdirect embedding, together with the fact  $\st_G(1)'\trianglelefteq\st_G(1)$, 
 we obtain 
 $([a^n,b_1],1,\dots,1) \in \psi(\st_G(1)')$ for any $n$. 
 Moreover, since $\st_G(1)'\trianglelefteq G$, we may conjugate the above element by a suitable power of $a$ to conclude that $(1, \dots,1, [a^n,b_1])\in \psi(\st_G(1)')$ for any $n$. 
 
 For any other directed generator $b_i$,
 $$\psi([b_1, b_i^a])=([a,b_i],1,\dots,1, [b_1,a^{e_{i,p-1}}]). $$
 Therefore, by the above, $([a,b_i],1,\dots,1)\in\psi(\st_G(1)')$. 
 Thus $(x, 1,\dots,1)\in \psi(\st_G(1)')$ for each normal generator $x$ of $G'$. 
 Once again the subdirect embedding $\psi(\st_G(1))\leq G\times \overset{p}{\dots}\times G$ allows us to conclude that 
$1\times \overset{p-1}{\dots}\times 1\times G' \leq \psi(\st_G(1)')$
 and since $G$ acts transitively on the first level of the tree we also have that $G'\times \overset{p}{\dots} \times G'\leq \psi(\st_G(1)')$. 
 
 Now suppose that all $b_i$'s are defined by symmetric vectors ($e_{i,j}=e_{i,p-j}$ for every $i,j\in \{1,\dots,p-1\}$).
  By (i) in Lemma \ref{multi-lemma} we may assume that $e_{i,1}=1$ for $i=1,\dots,r$. 
  Replacing each $\e_i$ by $\e_i-\e_1$ for $i=2,\dots,r$ we obtain the same group. 
  Thus we have $\e_i=(0,*,\ldots,*,0)$ for $i=2,\dots,r$.
  Let $e_{2,j}$ be  the first non trivial entry in $\e_2$. 
  Then there exists $k\in \F_p$ such that $e_{1,j}-ke_{2,j}=0$, so we can also replace $\e_1$ by $\e_1-k\e_2\neq 0$,  so that $e_{1,j}=0$.
  We thus obtain that
 $$\psi([b_1,b_i^a])=([a,b_i],1,\ldots,1),$$
 for $i=2,\ldots,r$, and
 \begin{align*}
 \psi([b_1^{a^j},b_2]) &=  (1,\dots,1,[b_1,a^{e_{2,j}}],1,\dots,1,[a^{e_{1,p-j}},b_2]) \\
& =(1,\dots,1,[b_1,a^{e_{2,j}}],1,\dots,1,1),
 \end{align*}
 where the last equality follows because $e_{1,p-j}=e_{1,j}=0$.
 Repeating the same argument as in the previous case we obtain the result.
\end{proof}

 Because of the above lemma (and Proposition 2.4 of \cite{AGU}),
 in order to show that a multi-GGS-group $G$ as in the lemma has the congruence subgroup property,  
  it suffices to show that $G''$ contains some level stabilizer. 
  This will be shown in Corollary \ref{cor}.

 \begin{lemma}\label{stabgamma}
  Let $G$ be any multi-GGS-group. Then $\st_G(1)'\leq \gamma_3(G)$.
 \end{lemma}
\begin{proof}
 Since $\st_G(1)$ is normally generated by $b_1,\dots,b_r$ (equivalently, generated by the conjugates of $b_1,\dots,b_r$ by powers of $a$), 
 we have that $\st_G(1)'$ is normally generated by commutators of the form $[b_i^{a^m}, b_j^{a^n}]$ with $i, j\in \{1,\dots,r\}$ and $m,n \in \F_p$. 
 Now notice that $[b_i^{a^m}, b_j^{a^n}]=[b_i[b_i,a^m], b_j[b_j,a^n]]$ which is congruent modulo $\gamma_3(G)$ to $[b_i, b_j]=1$. 
 Thus all normal generators of $\st_G(1)'$ are contained in $\gamma_3(G)\trianglelefteq G$, which proves our claim. 
\end{proof}

 \begin{lemma}\label{subdirect}
Let $G\neq \mathcal{G}$ be a multi-GGS-group. 
 Then
 $$\psi(G')\leq_s G\times\overset{p}{\dots}\times G.$$
\end{lemma}
\begin{proof}

Let $G= \langle a,b_1,\ldots,b_r\rangle$ a multi-GGS group and without loss of generality, we may assume that all defining vectors are non-constant. Let $G_i = \langle a, b_i \rangle$ be the associated GGS-group for each $1 \le i \le r$. From \cite[Lemma 2.5]{AGU}, we have $\psi(G_i')\le_s G_i \times\ldots\times G_i$ for $1 \le i \le r$. Hence $\psi (G') \le_s G \times\ldots\times G$, as $G$ is generated by $a, b_1, \ldots, b_r$.

\end{proof}

\begin{lemma}\label{psi2}
Let $G=\langle a,b_1,\dots,b_r\rangle$ be a multi-GGS-group with $r\geq 2$.
 Then $$\psi_2(G'')\geq G'\times\overset{p^2}{\dots}\times G'.$$
\end{lemma}
\begin{proof}
By Lemma \ref{subdirect} we know that there exist $x,y_i\in G'$ such that $\psi(x)=(a,*,\dots,*)$ and $\psi(y_i)=(b_i,*,\dots,*)$ for each $i\in\{1,\dots,r\}$ (where $*$ denotes unknown, unimportant elements).
 On the other hand, by Lemma \ref{regularbranch}, for each $h\in G'$ there is some $g\in G'$ such that $\psi(g)=(h,1,\dots,1)$. Thus $\psi([x,g])=([a,h],1,\dots,1)$ and $\psi([y_i,g])=([b_i,h],1,\dots,1)$ for $i=1,\dots,r$. 
 Now, $[x,g],[y_i,g]\in G''$ implies that $$\psi(G'')\geq \gamma_3(G)\times\overset{p}{\dots}\times\gamma_3(G).$$
Finally, Lemma \ref{stabgamma} and another application of Lemma \ref{regularbranch} yield the result.
\end{proof}

\section{Proof of the main Theorem}

Theorem \ref{thm:CSP} follows from Theorems A and B in \cite{AGU} (the GGS-group case) and from the result in this section.

Let us first establish some notation.
For all $n\in \mathbb{N}$, set 
$$G_n=\frac{G}{\st_G(n)}, \qquad \overline{G}_n=\frac{G_n}{G_n'}, \qquad  \widehat{G}_n=\frac{G_n}{\st_{G_n}(1)'},$$
and write $G_0=1$. Observe that in the same way in which $\psi:\st_G(1)\rightarrow G\times\overset{p}{\dots}\times G$  holds, we also have $\psi_{(n)}:\st_{G_n}(1)\rightarrow G_{n-1}\times\overset{p}{\dots}\times G_{n-1}$. Denoting by $\pi_n$ the projection from $G$ to $G_n$, the following diagram commutes:

%

$$
\begin{tikzcd}
 \st_G(1) \arrow{r}{\psi} \arrow{d}{\pi_n} & G\times\overset{p}{\dots}\times G \arrow{d}{\pi_{n-1}\times\overset{p}{\dots}\times\pi_{n-1}} \\
 \st_{G_n}(1) \arrow{r}{\psi_{(n)}}  & G_{n-1}\times\overset{p}{\dots}\times G_{n-1}
\end{tikzcd}
$$

Moreover, since $\psi:\st_G(1)'\longrightarrow G'\times\overset{p}{\dots}\times G'$ is an isomorphism, the map $$\widehat{\psi}_{(n)}:\frac{\st_{G_n}(1)}{\st_{G_n}(1)'}\longrightarrow \overline{G}_{n-1}\times\overset{p}{\dots}\times\overline{G}_{n-1}$$ is well defined. 
\begin{proposition}\label{stab}
Let $G=\langle a,b_1,\dots,b_r\rangle$ be a multi-GGS-group.
 Then $G'\geq \st_G(r+1)$.
\end{proposition}
\begin{proof}
We will prove by induction on $n\in \N$ that $d(\overline{G}_n)\geq n$ for $n=2,\dots,r+1$. This implies in particular that $\overline{G}_{r+1}$ is generated by $r+1$ elements, and then $|\overline{G}_{r+1}|=| G: G'|$, which implies that $G'=G'\st_G(r+1)$ and the result follows.

Observe that $d(G_n)=d(\overline{G}_n)=d(\widehat{G}_n)$, because $G_n^p\leq G_n'$ and then $\Phi(G_n)=G_n'$. Since $G_n'$ and $\st_{G_n}(1)'$ are contained in $\Phi(G_n)$, the minimum number of generators does not change.

The case $n=2$ is obvious because if $G_2$ is generated by one element, then $\st_G(1)=\st_G(2)$ and this cannot happen.
 Let us suppose the statement is true for $n\leq r$, that is $d(\overline{G}_n)\geq n$.
  Since $\overline{G}_n$ is elementary abelian, and it is generated by the projections of the generators of $G$, we can choose a basis and we may assume that $\overline{G}_n=\langle \overline{a},\overline{b}_1,\dots,\overline{b}_{n-1},\dots\rangle$ where the first $n$ generators are linearly independent in $\overline{G}_n$.
   We want to prove the case $n+1$. 
   Suppose for a contradiction that $\widehat{G}_{n+1}$ can be generated by $n$ elements.
In order to generate $\widehat{G}_{n+1}$, we need some element congruent to $\widehat{a}$ modulo $\st_{G_{n+1}}(1)'$.
 On the other hand, by the Burnside Basis Theorem, since 
$\st_{G_{n+1}}(1)/G_{n+1}'$ has rank at most $n-1$ because we assumed that $d(\widehat{G}_{n+1})\leq n$, we can choose a basis $\widehat{b}_1,\dots,\widehat{b}_{n-1}$ of $\st_{G_{n+1}}(1)/\st_{G_{n+1}}(1)'$ (these elements are linearly independent because they map onto $\overline{b}_1,\dots,\overline{b}_{n-1}$, respectively, which are assumed to be linearly independent). 
We may thus suppose that $\widehat{G}_{n+1}=\langle \widehat{a},\widehat{b}_1,\dots,\widehat{b}_{n-1}\rangle$.
%
%
%
%
Then
$$	\widehat{b}_n=\widehat{b}_1^{i_{1,0}}(\widehat{b}_1^{\widehat{a}})^{i_{1,1}}\dots(\widehat{b}_1^{\widehat{a}^{p-1}})^{i_{1,p-1}}\dots
 \widehat{b}_{n-1}^{i_{n-1,0}}(\widehat{b}_{n-1}^{\widehat{a}})^{i_{n-1,1}}\dots(\widehat{b}_{n-1}^{\widehat{a}^{p-1}})^{i_{n-1,p-1}}$$
with $i_{j,k}\in\F_p$.
But then the images under $\widehat{\psi}_{(n)}$ of the element on the left-hand side and right-hand side must be equal in $\overline{G}_n$. Since $\widehat{\psi}_{(n)}(\widehat{b}_n)=(\overline{a}^{e_{n,1}},\dots,\overline{a}^{e_{n,p-1}},\overline{b}_n$) we are forced to have $i_{j,k}=0$ for $k\neq 0$.
 This means that $e_{n}=i_{0,1}e_1+\dots+i_{0,n}e_{n-1}$, which is impossible, because the defining vectors are linearly independent.
  Thus, $d(\widehat{G}_{n+1})\geq n+1$ and the theorem follows by induction.
\end{proof}

\begin{cor}\label{cor}
	Let $G$ be as in Proposition \ref{stab}. 
	Then $G$ has the congruence subgroup property. 
\end{cor}
As remarked previously, it suffices to show, by Proposition 2.4 of \cite{AGU} and Lemma \ref{regularbranch}, that $G''$ contains some level stabilizer. 
Lemma \ref{psi2} and Proposition \ref{stab} yield that
 $$\psi_2(G'')\geq \st_G(r+1)\times\overset{p^2}{\dots}\times \st_G(r+1)\geq \psi_2(\st_G(r+3)).$$
Thus $G''\geq \st_G(r+3)$.


\begin{thebibliography}{10}
  \bibitem[AGU]{AGU} G.A. Fern\'andez-Alcober, A. Garrido and  J. Uria-Albizuri. \newblock On the congruence subgroup property for {GGS}-groups. \newblock {\em Proc. Amer. Math. Soc.} {\bf 145} (2017), 3311-3322.
  
 \bibitem[AKT]{AKT}T. Alexoudas, B. Klopsch and A. Thillaisundaram. \newblock Maximal subgroups of multi-edge spinal groups. \newblock {\em Groups Geom. Dyn.} {\bf 10} (2016), 619-648.

\bibitem[AZ14]{Alc}
G.A. Fern{\'a}ndez-Alcober and A. Zugadi-Reizabal.
\newblock G{GS}-groups: order of congruence quotients and {H}ausdorff dimension.
\newblock {\em Trans. Amer. Math. Soc.} {\bf 366} (2014), 1993--2007.

\bibitem[Gri80]{grigorchuk:burnside}
R.I. Grigorchuk.
\newblock On {B}urnside's problem on periodic groups.
\newblock {\em Funktsional. Anal. i Prilozhen.}, {\bf 14(1)} (1980), 53--54.

\bibitem[GS83]{gupta-sidki:burnside}
N. Gupta and S.N. Sidki.
\newblock On the {B}urnside problem for periodic groups.
\newblock In {\em Mathematische Zeitschrift}, {\bf 182} (1983), 385--388.
 \end{thebibliography}
\end{document}